\documentclass[11pt]{amsart}
\usepackage{amsbsy,amssymb,amscd,amsfonts,latexsym,amstext,delarray, amsmath,graphicx,color,caption, amsthm, enumerate, verbatim}
\usepackage{amstext,amsopn, mathtools, bbm}
\usepackage{graphicx}
\usepackage{mathtools}
\usepackage{hyperref}
\usepackage{enumitem}
\usepackage{epsfig}
\input xy

\newcommand\R{{\mathbb R}}

\hypersetup{
	colorlinks=true,
	linkcolor=black,
	filecolor=black,      
	urlcolor=black,
	citecolor = black
}

\begin{document}
	
	\newtheorem{example}{Example}[section]
	\newtheorem{lemma}{Lemma}[section]
	\newtheorem{thm}{Theorem}
	\newtheorem{prop}[lemma]{Proposition}
	\newtheorem{cor}{Corollary}[section]
	\newtheorem{conjecture}{Conjecture}[section]
	
	\theoremstyle{remark}
	\newtheorem{remark}{\textbf{Remark}}[section]

\title[Hyperbolic Restriction]{Restriction estimates for hyperbolic paraboloids in higher dimensions via bilinear estimates}
\author{Alex Barron}

\maketitle

\newcommand{\Addresses}{{
		\bigskip
		\footnotesize

		\textsc{Department of Mathematics, University of Illinois Urbana-Champaign,
			Urbana, IL 61801, USA}\par\nopagebreak
		\textit{E-mail address}: \texttt{aabarron@illinois.edu}

}}

\begin{abstract} Let $\mathbb{H}$ be a $(d-1)$-dimensional hyperbolic paraboloid in $\mathbb{R}^d$ and let $Ef$ be the Fourier extension operator associated to $\mathbb{H},$ with $f$ supported in $B^{d-1}(0,2)$. We prove that $\|Ef\|_{L^p (B(0,R))} \leq C_{\epsilon}R^{\epsilon}\|f\|_{L^p}$ for all $p \geq \frac{2(d+2)}{d}$ whenever $ \frac{d}{2} \geq m + 1$, where $m$ is the minimum between the number of positive and negative principal curvatures of $\mathbb{H}$. Bilinear restriction estimates for $\mathbb{H}$ proved by S. Lee and Vargas play an important role in our argument.  
\end{abstract}

\vspace{5mm}

In this paper we study estimates for the operator $$Ef(x,t) = \int_{\R^{d-1}} f(\xi) e^{2\pi i (x\cdot \xi + t(\xi_{1}^2 + ... + \xi_{d-m-1}^2 - \xi_{d-m}^2 - \xi_{d-m+1}^2 - ... - \xi_{d-1}^2 ) ) } d\xi,$$ $$\text{supp}(f) \subset B^{d-1}(0,2).$$ This is the extension operator associated to the hyperbolic paraboloid $$\mathbb{H} = \{\xi \in \R^d : \xi_d =\xi_{1}^2 + \xi_{2}^2 + ... + \xi_{d-m-1}^2 - \xi_{d-m}^2 - ... - \xi_{d-1}^2 \}.$$  We let $M$ denote the $(d-1)\times(d-1)$ diagonal matrix with $M_{ii} = 1$ if $i \leq d- 1 - m$ and $M_{ii} = -1$ if $i > d-1-m$. Then the phase of $Ef$ has the form $$x\cdot \xi + t(M\xi \cdot \xi), \ \ \ \xi \in \R^{d-1}.$$ We can assume that $m \leq \frac{d-1}{2}$ since otherwise we can replace $t$ by $-t$. Note that $m$ is the minimum between the number of positive and negative principal curvatures of $\mathbb{H}$.

We will prove the following. \begin{thm}\label{thm:restrictionThm} Fix $d \geq 4$. Suppose $f$ is supported in $B^{d-1}(0,2)$ and fix $R \geq 1$ and $\epsilon > 0$. If $m \leq \frac{d}{2}  - 1$ and $p \geq \frac{2(d + 2)}{d}$ then \begin{equation} \label{eq:mainEst} \|Ef\|_{L^{p}(B_R) } \leq C_{\epsilon} R^{\epsilon}\|f\|_{L^p}.\end{equation} 
\end{thm} \noindent When $d$ is even this result follows from work of Bourgain--Guth (\cite{BG}, see Remark \ref{rmk:BG}). By Tao's $\epsilon$-removal argument (\cite{T2}) the theorem holds for $p > \frac{2(d+2)}{d}$ with no loss of $R^{\epsilon}.$ 

There is an alternative proof of Theorem \ref{thm:restrictionThm} that can be found in the recent paper \cite{HiI} by Hickman and Iliopoulou, which is discussed further below. Indeed, in the cases where $m < \lceil \frac{d-3}{2}\rceil $ the estimates in \cite{HiI}, which generalize the polynomial partitioning method of \cite{G} and \cite{GHI}, are stronger than Theorem \ref{thm:restrictionThm}. The main novelty of our approach will be the use of bilinear estimates of S. Lee and Vargas (see below) and an elementary orthogonality estimate in place of $k$-linear restriction estimates \cite{BCT} and the $\ell^2$ decoupling theorem of Bourgain and Demeter \cite{BD2}. This gives a somewhat more elementary proof of Theorem \ref{thm:restrictionThm}, which matches the best known restriction estimate for the hyperbolic paraboloid in the case where $d$ is odd and $m \geq \frac{d-1}{2} - 1$, and where $d$ is even and $m = \frac{d}{2} -1$. Moreover, it was previously unknown if the bilinear estimates we use implied linear estimates in dimension $d \geq 4$. Some of the geometric observations we take advantage of in our proof may also be useful for future work on restriction estimates for $\mathbb{H}$, in particular in the signature 0 case where $d$ is odd and $m = \frac{d-1}{2}$. In this case the Stein--Tomas theorem is the best known restriction estimate for $\mathbb{H}$.      

\subsection*{Related results in the literature} In the case $d = 3, m = 1$ Theorem 1 was proved independently by Vargas (\cite{V}) and S. Lee (\cite{L}) using a bilinear method. This was later improved by Cho and J. Lee (\cite{CL}), who adapted the polynomial partitioning method developed by Guth in \cite{G1} to show that \eqref{eq:mainEst} holds for $p > 3.25$. In \cite{S} Stovall proved certain endpoint cases when $d = 3$ that do not follow from arguments in \cite{L} and \cite{V}. See also the paper \cite{K} by Kim. For other recent progress on restriction estimates for perturbations of the hyperbolic paraboloid in dimension 3 see the recent papers of  Buschenhenke-M\"{u}ller-Vargas \cite{BMV} and Guo-Oh \cite{GO}. 

When $d \geq 4$ the bilinear-to-linear reduction applied by Vargas and S. Lee breaks down for reasons we discuss further in Section 2.1. Improved restriction estimates also do not follow immediately from the techniques established by Guth in \cite{G} to study elliptic paraboloids in higher dimensions. Indeed, the transverse equidistribution estimates that play a crucial role in Guth's argument can fail for hyperbolic paraboloids in certain cases (see Example 8.8 in \cite{GHI}). For related reasons the Bourgain--Guth method developed in \cite{BG} also does not easily adapt to hyperbolic paraboloids in the case where $d \geq 5$ is odd, although if $d$ is even then the estimate in Theorem \ref{thm:restrictionThm} follows from their more general estimates for H\"{o}rmander-type operators  (see Remark \ref{rmk:BG} at the end of Section 3 below).

As mentioned above, Hickman and Iliopoulou \cite{HiI} have recently extended the polynomial partitioning method developed by Guth in \cite{G} and Guth, Hickman, and Iliopoulou in \cite{GHI} to the hyperbolic case. The key new ingredient is the introduction of certain weakened transverse equidistribution estimates. These estimates, which describe the extent to which $Ef$ can concentrate along the neighborhood of a lower-dimensional variety, get worse as the parameter $m$ increases but are still strong enough to obtain improved restriction bounds when $m$ is not too large.

Certain sharp fractal estimates for $Ef$ have also been recently obtained by the author, Erdo\u{g}an, and Harris in \cite{BEH}. These estimates extend the fractal restriction argument of Du and Zhang \cite{DZ} to the hyperbolic case. A weighted version of the bilinear argument in this paper plays a key role in the proof of some of the sharp results in the main theorem in \cite{BEH}.   

\subsection*{Overview of the paper}The main goal for this paper is to prove Theorem \ref{thm:restrictionThm} using the bilinear restriction estimates proved by S. Lee and Vargas, stated precisely in Theorem \ref{thm:bilinear0} in Section 2 below. It was previously unknown if these bilinear estimates could be used to prove linear estimates in the range $p\geq \frac{2(d+2)}{d}$, due to a number of geometric obstructions that arise when trying to apply the usual bilinear-to-linear method in dimension $d \geq 4$. Our argument will follow a broad-narrow scheme adapted from \cite{BG}, \cite{DZ}, \cite{G}. This broad-narrow analysis allows us to use the estimates of S. Lee and Vargas except in certain exceptional cases which we analyze in Section 2. The main idea is the following: if $\tau_1$ and $\tau_2$ are two caps in the support of $f$ and we do not have a favorable estimate for $Ef_{\tau_1}Ef_{\tau_2},$ then $\tau_1$ and $\tau_2$ must be arranged in a neighborhood of a hyperbolic cone $\mathcal{C}_{m}$. If we can find no pairs $(\tau_1, \tau_2)$ for which bilinear estimates apply then the geometry of $\mathcal{C}_{m}$ forces the caps to in fact be contained in a neighborhood of an $m$-dimensional plane; we can then treat this scenario using a `narrow'  flat decoupling argument and induction on the scale, at least when $m \leq \frac{d}{2} - 1$. In the special case where $d$ is odd and $m = \frac{d-1}{2}$ our method breaks down since the induction no longer closes. Note however that we always have $m \leq \frac{d-1}{2}$.

We review some basic tools that we will use frequently in Section 1. In Section 2 we discuss some more history and background surrounding bilinear restriction estimates. The key lemma describing how bilinear estimates for $Ef$ can fail is then proved in Section 2.3. Finally, in Section 3 we carry out the broad-narrow argument to complete the proof of Theorem 1. Some remarks about the failure of our argument in the case where $d$ is odd and $m = \frac{d-1}{2}$ can be found at the end of Section 3. 

\subsection*{Notation} We will write $A \lesssim B$ if there is some constant $c > 0$ depending only on the dimension and various Lebesgue exponents such that $A \leq cB$. If $A \lesssim B$ and $B \lesssim A$ we also write $A \sim B$. Our uniform constants may also vary from line-to-line, which is allowed since they will remain independent of $R$. We will also write $A \lessapprox B$ to signify that for each $\epsilon > 0$ there is some $C(\epsilon)$ such that $A \leq C(\epsilon)R^{\epsilon} B$. Finally, we will write $A \ll B$ if $A/B \rightarrow 0$ as $R\rightarrow \infty$. 

Let $B_{r}$ be a ball of radius $r$ in $\R^d$ and let $B_{r^{-1}}$ denote a ball centered at the origin in $\R^d$ of radius $r^{-1}$. We let $w_{B_r}$ be a smooth weight adapted to $B_r$ in the following sense: $w_{B_r}(x,t)$ decays rapidly for $(x,t) \notin B_r$, and $\widehat{w_{B_r}}$ is supported in a fixed dilate of $B_{r^{-1}}$. Note that we can construct $w_{B_r}$ by taking a bump function $w$ adapted to the unit ball such that $$|w(x)| \lesssim \frac{1}{(1 + |x|)^{1000d}}$$ and then applying a suitable affine transformation.

If $S$ is a ball or rectangle in $\R^{d-1}$ we let $f_{S} = f\cdot \phi_{S},$ where $\phi_{S}$ is a smooth bump function supported in a small dilate of $S$ with $\phi_{S} (\xi) = 1$ when $\xi \in S$. If $\mathcal{M}$ is a smooth manifold and $\rho > 0$ we will let $N_{\rho}(\mathcal{M})$ denote the $\rho$-neighborhood of $\mathcal{M}$. 

\subsection*{Acknowledgments} This paper benefited from several helpful conversations with M. Burak Erdo\u{g}an and Terence Harris. The author also thanks the anonymous referees for helpful comments which have improved the paper.  

\section{Basic tools}

In this section we review some basic tools we will use throughout the proof of Theorem \ref{thm:restrictionThm}. Below we will always assume that the support of $f$ is contained in $B^{d-1}(0,2)$. 

\subsection{Wave packet decomposition and parabolic rescaling} We first recall the standard wave packet decomposition for $Ef$ (see for example \cite{CL}, \cite{G}, \cite{L}, or \cite{V}). Fix $\rho\in (0,1)$ and suppose $\{\tau \}$ is a collection of finitely-overlapping balls of radius $\rho$ that cover the support of $f$. We will refer to these $\tau$ as \textit{$\rho$-caps}. Using a partition of unity we may decompose $f = \sum_{\tau} f_{\tau}$, with $f_{\tau}$ supported in a small dilate of $\tau$. Then $Ef = \sum_{\tau}Ef_{\tau}$. We let $$G(\tau) := \frac{(2\xi_1, ...,2\xi_{d-m-1}, -2\xi_{d-m},...,-2\xi_{d-1}, -1)}{|(2\xi_1,...,2\xi_{d-m-1}, -2\xi_{d-m},...,-2\xi_{d-1}, -1)|}$$ when $\xi$ is the center of $\tau$, so $G(\tau)$ is the unit normal direction to $\mathbb{H}$ above the center of $\tau$. If $T_{\tau}$ is any tube in $\R^d$ of dimensions $$\rho^{-1}\times...\times \rho^{-1} \times \rho^{-2}$$ with long direction $G(\tau)$ then $Ef_{\tau}$ is essentially constant on $T_{\tau}$. 

We also recall that $Ef$ is invariant under parabolic rescalings in the following way.  

\begin{prop}\label{prop:parabolicRescaling} Fix $R > 1$ and let $B_R = B^{d}(0,R)$. Also fix $\rho \in (0,1)$ with $\rho^{-1} < R.$ Then for any $\rho$-cap $\tau$ one can find a function $g$ supported in $B^{d-1}(0,2)$ such that $$\|g\|_{L^{p}} = \rho^{-\frac{d-1}{p}}\|f_{\tau}\|_{L^{p}} $$ and $$\|Ef_{\tau}\|_{L^{p}(B_R)} \leq \rho^{(d-1) - \frac{d+1}{p} }\|Eg\|_{L^{p}(B_{\rho R})}.$$ 
\end{prop}

\noindent To prove the proposition one can use modulation invariance of $Ef_{\tau}$ to reduce to the case where $\tau$ is centered at the origin, and then rescale $(x,t) \rightarrow (\rho^{-1} \bar{x}, \rho^{-2} \bar{t}).$ 

The operator $Ef$ has other scaling symmetries that differ from parabolic rescaling, although we will make no use of these symmetries in our arguments. Note, however, that the proof of Theorem \ref{thm:restrictionThm} in the case $d= 3$ due to S. Lee and Vargas (\cite{L}, \cite{V}) does exploit these extra symmetries. The same is also true of the Bourgain--Guth proof of the case $d = 3$ (see Remark \ref{rmk:BG} below), along with the improved estimate when $d = 3$ due to Cho and J. Lee in \cite{CL}.

\subsection{Flat decoupling and induction on scales} Decoupling allows us to separate the contribution from different wave packets $Ef_{\tau}$. This is useful in the `narrow case' below when we cannot use bilinear restriction estimates. The strongest possible decoupling result for the hyperbolic paraboloid has been proved by Bourgain and Demeter (\cite{BD2}), though we will not need to use their theorem. Instead it will suffice to use the following more elementary `flat decoupling' result, which follows easily from orthogonality considerations. 

\begin{prop}[Flat Decoupling] \label{prop:decoupling} Suppose $\mathcal{T}$ is a collection of finitely-overlapping $\rho$-caps $\tau$ with $\rho^{-1} < R$. Then $$\| \sum_{\tau \in \mathcal{T} } Ef_{\tau}\|_{L^{p}(B_R)} \leq C(\# \mathcal{T})^{\frac{1}{2} - \frac{1}{p}}  \big( \sum_{\tau \in \mathcal{T} } \|Ef_{\tau}\|^{2}_{L^{p} (w_{B_R}) } \big)^{\frac{1}{2}},$$ where $w_{B_R}$ is a smooth weight adapted to $B_R$. 
\end{prop}

\begin{proof} The case $p = \infty$ is just the Cauchy-Schwarz inequality, and when $p =2$ the proposition follows from Plancharel's theorem. The remaining cases follow by interpolation. 
\end{proof}

Finally we recall that if $R$ is small enough then Theorem \ref{thm:restrictionThm} follows directly from H\"{o}lder's inequality. We can therefore assume by induction that Theorem \ref{thm:restrictionThm} is true at scale $\rho R$ whenever $\rho \ll 1$. For technical reasons related to the decoupling result in Proposition \ref{prop:decoupling} we also remark that we can assume by induction that the following weighted estimate holds: for any $\epsilon > 0 $ $$\|Ef\|_{L^p (w_{B_{\rho R}} )} \leq C_{\epsilon}(\rho R)^{\epsilon}\|f\|_{L^{p}},$$ where $w_{B_{\rho R}}$ is a smooth weight adapted to $B_{\rho R}.$ 

\section{Bilinear restriction estimates for $\mathbb{H}$} In this section we will review some known bilinear estimates and prove a lemma that characterizes what happens if these bilinear estimates fail. The following estimate was proved by S. Lee in dimension $d \geq 3$ (\cite{L}) and independently proved by Vargas in dimension 3 (\cite{V}).

\begin{thm}[\cite{L}, \cite{V}]\label{thm:bilinear0}  Suppose $f_1$ and $f_2$ are supported in open sets $\tau_1$ and $\tau_2$ of diameter $\sim 1$.  If \begin{equation}\label{eq:bilinearCondition0} \inf_{ \substack{ \xi, \bar{\xi} \in \tau_1 \\ \eta, \bar{\eta} \in \tau_2 } } | M(\xi - \eta) \cdot (\bar{\xi} - \bar{\eta})| \geq c > 0\end{equation} then \begin{equation}\label{eq:bilnearEst0} \| |Ef_1 Ef_{2}|^{\frac{1}{2}} \|_{L^{p}(B_R)} \leq C_{\epsilon}R^{\epsilon} \|f_1\|^{\frac{1}{2}}_{L^2} \|f_2\|_{L^2}^{\frac{1}{2}} \end{equation} whenever $p \geq \frac{2(d+2)}{d}$. If \eqref{eq:bilinearCondition0} fails then \eqref{eq:bilnearEst0} can fail as well for all $p \geq \frac{2(d+2)}{d}$.  
\end{thm}

\noindent We will need to use a version of Theorem \ref{thm:bilinear0} adapted to $K^{-1}$-caps for a parameter $K$ such that $$1 \ll K \ll R.$$ The following is a consequence of Theorem \ref{thm:bilinear0}.

\begin{thm}\label{thm:bilinear}  Suppose $f_1$ and $f_2$ are supported in $K^{-1}$-caps $\tau_1$ and $\tau_2$, respectively, whose centers are separated by $CK^{-1}$, with $C$ a sufficiently large but uniform constant. Let $A$ be a constant. If \begin{equation}\label{eq:bilinearCondition} \inf_{ \substack{ \xi, \bar{\xi} \in \tau_1 \\ \eta, \bar{\eta} \in \tau_2 } } | M(\xi - \eta) \cdot (\bar{\xi} - \bar{\eta})| \geq AK^{-1}\end{equation} then \begin{equation}\label{eq:bilnearEst} \| |Ef_1 Ef_{2}|^{\frac{1}{2}} \|_{L^{p}(B_R)} \leq C_A K^{O(1)} \|f_1\|^{\frac{1}{2}}_{L^2} \|f_2\|_{L^2}^{\frac{1}{2}} \end{equation} whenever $p \geq \frac{2(d+2)}{d}$. If \eqref{eq:bilinearCondition} fails then \eqref{eq:bilnearEst} can fail as well for all $p \geq \frac{2(d+2)}{d}$.
\end{thm}

\noindent We say that two $K^{-1}$-caps $\tau_1, \tau_2$ are \textit{strongly separated} if \eqref{eq:bilinearCondition} holds. 

Since it is not immediately obvious from scaling that Theorem \ref{thm:bilinear0} implies Theorem \ref{thm:bilinear}, we will prove the implication below in Section 2.2.

\subsection{Some background} Bilinear restriction estimates in the full range given in Theorem \ref{thm:bilinear0} were first proved by Wolff in the case of the cone \cite{W}. Wolff's methods were later adapted by Tao in the case of the elliptic paraboloid \cite{T}, and then by Vargas and S. Lee independently in the case of hyperbolic paraboloids. In the case of the cone and the elliptic paraboloid the transversality condition \eqref{eq:bilinearCondition} is much simpler. 

There is an argument due to Tao, Vargas, and Vega (\cite{TVV}) that allows one to deduce linear restriction estimates from bilinear restriction estimates for elliptic surfaces, and indeed linear restriction estimates are obtained as corollaries of the main results in \cite{T} and \cite{W}. Let $E_{e}$ denote the extension operator associated to the elliptic paraboloid The main idea of the argument from \cite{TVV} is that any two points will belong to a unique pair of dyadic cubes that are separated by a distance proportional to their scale; one can then use this observation to efficiently decompose $|E_{e}f|^2$ as a sum of terms to which bilinear estimates apply (after a parabolic rescaling). For hyperbolic paraboloids this argument requires different ideas since the stronger transversality condition \eqref{eq:bilinearCondition} is more complicated. 

In the special case $d = 3, m =1 $ one can apply a simple change variables and instead consider the extension operator associated to the surface $$ \{\xi \in \R^3: \xi_{3} = \xi_1 \xi_2, \ \ |\xi| \leq 2 \}.$$ Then \eqref{eq:bilinearCondition} is equivalent to the following two-parameter separation condition: \begin{equation} \label{eq:threeDimTransv}|\xi_{1} - \eta_{1}| \gtrapprox 1 \ \ \ \text{ and } \ \ \ |\xi_{2} - \eta_{2}| \gtrapprox 1 \ \ \ \text{ for all  } \xi \in \tau_1, \eta \in \tau_{2}.\end{equation} Vargas and S. Lee were able to use this observation to almost recover the bilinear-to-linear reduction from \cite{TVV}, up to certain endpoint cases which were later proved by Stovall \cite{S}. All of these arguments rely on the fact that \eqref{eq:threeDimTransv} facilitates a two-parameter decomposition of frequency space analogous to the decomposition used in \cite{TVV}. When $d \geq 4$  this is no longer the case, and the condition \eqref{eq:bilinearCondition} is no longer well-adapted to Whitney-type decompositions. In particular note that if $d = 3$ then \eqref{eq:threeDimTransv} can only fail if all the caps are arranged in a neighborhood of an axis-parallel line (which becomes a diagonal or anti-diagonal line if we undo the change of variables and write the phase as $\xi_{1}^2 - \xi_{2}^2$). However, when $d \geq 4$ the estimate \eqref{eq:bilnearEst} can fail if the caps are contained near a subset of (a translation of) the hyperbolic cone $$\mathcal{C} = \{\xi \in \R^{d-1} : \xi_{1}^2 + ... + \xi_{d-m-1}^2 = \xi_{d-m}^2 + ... + \xi_{d-1}^2  \}.$$ 

After we deduce Theorem \ref{thm:bilinear} we will analyze what can happen in the exceptional case where \eqref{eq:bilinearCondition} fails for all pairs of caps in the support of $f$. We will see that failure of \eqref{eq:bilinearCondition} for every pair of caps forces $f$ to be supported near an affine space of dimension $m$. We will then be able to use decoupling and induction to prove Theorem \ref{thm:restrictionThm} in the `narrow' cases where we cannot use Theorem \ref{thm:bilinear}.

As mentioned in the introduction, our methods do not work when $d = 3, m = 1$. In this case Theorem \ref{thm:restrictionThm} is still true and follows from arguments by S. Lee, Stovall, or Vargas (\cite{L}, \cite{S}, \cite{V}). Of course when $d=3$ Theorem \ref{thm:restrictionThm} also follows from the stronger restriction estimate due to Cho and J. Lee \cite{CL}.

\subsection{Proof that Theorem \ref{thm:bilinear0} implies Theorem \ref{thm:bilinear}} Let $e_j$ denote the standard basis vectors in $\R^{d-1}$. Let $\tau_1$ and $\tau_2$ be two $K^{-1}$-caps for which \eqref{eq:bilinearCondition} holds. After translation we can assume that $\tau_2$ is centered at the origin. We may assume that dist$(\tau_1, \tau_2) \gtrsim K^{-\frac{1}{2}}$ since otherwise the desired result follows easily by rescaling frequency space by $K^{\frac{1}{2}}$. 

Since $\tau_2$ is centered at the origin the condition \eqref{eq:bilinearCondition} is invariant under linear transformations of the form $U = U' \oplus U'',$ where $U'$ is a rotation in $\xi_1,...,\xi_{d-m-1}$ that fixes $\xi_{d-m},...,\xi_{d-1},$ and $U''$ is a rotation in $\xi_{d-m},...,\xi_{d-1}$ that fixes $\xi_{1},...,\xi_{d-m-1}$. We can therefore assume that $\tau_1$ is centered at a point of the form $$\xi^{\ast} = (\xi_1,0,...,0,\xi_{d-m},\xi_{d-m+1},...,\xi_{d-1})$$ with $$|\xi_{1}^2 - \xi_{d-m}^{2} - ... - \xi_{d-1}^2 | \geq cK^{-1}.$$ Let us first assume that $$\xi_{1}^2 - \xi_{d-m}^{2} - ... - \xi_{d-1}^2 \geq cK^{-1}.$$ Since we are also assuming $|\xi^{\ast}|^2 \geq cK^{-1}$ it follows that \begin{equation} \label{eq:detEq} \xi_{1}^2 \geq cK^{-1}.\end{equation} Now let $S$ be the linear transformation such that $$S\xi^{\ast} = e_1$$ $$S e_{j} = e_{j} \ \ \ j= 2,..., d-1.$$ One checks using \eqref{eq:detEq} that $$\|S\| \sim \frac{1}{|\xi_1|} \lesssim K^{\frac{1}{2}}.$$ In particular the first column of $S$ is $$(1/\xi_1, 0,...,0, -\xi_{d-m}/\xi_1,...,-\xi_{d-1}/\xi_1)$$ while the other columns are $e_2,...,e_{d-1}$. Now suppose $\eta, \bar{\eta} \in \tau_2$. Since we are assuming that $\tau_2$ is centered at the origin we then have

\begin{align} \label{eq:detEq2} |M( S\xi^{\ast} - S\eta )  \cdot (S\xi^{\ast} - S\bar\eta )| &= |Me_1 \cdot e_1 + O(K^{-\frac{1}{2}})| \gtrsim 1.  \end{align} Since changing $\xi^{\ast}$ to any other $\xi \in \tau_1$ in \eqref{eq:detEq2} only introduces an error of $O(K^{-\frac{1}{2}})$ it follows that the caps $S\tau_1, S\tau_2$ satisfy the condition \eqref{eq:bilinearCondition0}, and so \eqref{eq:bilnearEst} follows from \eqref{eq:bilnearEst0} after rescaling $f_1, f_2$ (which is allowed since we can lose $K^{O(1)}$ in the bilinear estimate).

In the case where $$-\xi_{1}^2 + \xi_{d-m}^{2} + ... + \xi_{d-1}^2 \geq cK^{-1}$$ we apply another transformation $U = U' \oplus U''$ to map $\xi^{\ast}$ to $$U\xi^{\ast} = (\xi_1, 0, ..., 0, \tilde{\xi}_{d-m}, 0,...,0 ).$$ Then since $U$ does not change the norm of either $(\xi_{d-m}, \xi_{d-m+1},...,\xi_{d-1})$ or $(\xi_1, 0,...,0)$ it follows that $$\tilde{\xi}_{d-m} - \xi_1 \geq cK^{-1},$$ and so we repeat the previous argument with $\tilde{\xi}_{d-m}$ playing the role of $\xi_1$. 

\subsection{Failure of bilinear estimates.} We now prove that if the bilinear estimates in Theorem \ref{thm:bilinear} fail then the caps $\tau$ must be localized near an $m$-dimensional plane. We first prove some geometric lemmas that will lead us in this direction, with the main result of the section being Lemma \ref{lem:caseLemma} below. Given $\xi = (\xi_1, ..., \xi_{d-1}) \in \R^{d-1}$, we will write $\xi' = (\xi_1, \xi_2,...,\xi_{d-m-1})$ and also $\xi'' = (\xi_{d-m},...,\xi_{d-1}).$ The following lemma can be thought of as an approximate polarization identity.

\begin{lemma}\label{lem:coneLemma} Let $\mathcal{C}$ denote the surface $$\mathcal{C}= \{ \xi \in B^{d-1}(0,2) :  \xi_{1}^2 + ... + \xi_{d-m-1}^2 = \xi_{d-m}^2 + ... + \xi_{d-1}^2, \ \ |\xi| < 2 \}$$ and let $$\mathcal{C}_{r} = \{\xi\in B^{d-1}(0,2) : \ |\xi \cdot M \xi| \leq r  \}. $$ Suppose $\xi, \eta \in \mathcal{C}_{cK^{-1}}.$ Let $T_{\xi}$ denote the subspace $$ T_{\xi} = \{\omega \in \R^{d-1} : \omega \cdot M\xi =0 \}.$$ If $\xi - \eta \in \mathcal{C}_{CK^{-1}}$ then $\eta$ is in an $O(K^{-1})$ neighborhood of $T_{\xi}$. \end{lemma}

\begin{proof} Since $\xi, \eta \in \mathcal{C}_{cK^{-1}}$ and $\xi - \eta \in \mathcal{C}_{CK^{-1}}$ we have \begin{align*}\sum_{i=1}^{d-m-1} (\xi_i - \eta_i)^2 &= \sum_{i=d-m}^{d-1} (\xi_i - \eta_i)^2 + O(K^{-1}) \\ &= \sum_{i=1}^{d-m-1} (\xi_{i}^2 + \eta_{i}^2) - 2\sum_{i=d-m}^{d-1} \xi_i \eta_i + O(K^{-1}).\end{align*}
	
\noindent After expanding the square on the left side and rearranging, we obtain \begin{equation*} \xi' \cdot \eta' = \sum_{i=d-m}^{d-1} \xi_i \eta_i + O(K^{-1}) = \xi'' \cdot \eta'' + O(K^{-1}). \end{equation*} As a consequence \begin{equation*} \eta \cdot M \xi = O(K^{-1}),
\end{equation*} which proves the lemma.
		
\end{proof}

\begin{lemma}\label{lem:tangentPlane} Let $V$ be a subspace of $\R^{d-1}$ and suppose that $V \cap B^{d-1}(0,2) \subset \mathcal{C}_{cK^{-a}},$ where $a > 0$. Then if $K$ is sufficiently large we must have dim $V \leq m$.
\end{lemma}

\begin{proof} Let $\{v^1, ..., v^k\}$ be an orthonormal basis for $V$. By hypothesis we know that $$v^{i}\cdot Mv^{i} = O(K^{-a})$$ for each $i$. Also note that $v^{i} - v^{j} \in V \cap B^{d-1}(0,2)$ and therefore $v^{i} - v^j \in \mathcal{C}_{cK^{-a}}.$ Then from Lemma \ref{lem:coneLemma} we conclude that \begin{equation}\label{eq:plane1} v^i \cdot Mv^j = O(K^{-a}) \end{equation} for each pair $i,j$. Of course \begin{equation}\label{eq:plane2}v^i \cdot v^j = 0, \ \ \ \ \ i \neq j\end{equation} by hypothesis. Now let $P: \R^{d-1} \rightarrow \R^{m}$ denote the orthogonal projection $$P\omega = (\omega_{d-m},...,\omega_{d-1}) \in \R^{m}.$$  

\noindent From \eqref{eq:plane1} and \eqref{eq:plane2} we conclude that \begin{equation} \label{eq:plane3} Pv^{i} \cdot Pv^{j} = O(K^{-a}) \ \text{ if} \ \ i \neq j, \ \ \ \ Pv^i \cdot Pv^i = \frac{1}{2} + O(K^{-a}). \end{equation} But if $K$ is large enough, depending only on $a$ and the implicit constants above, then \eqref{eq:plane3} implies that the set $\{Pv^1,..., Pv^{k} \}$ is linearly independent. One way to see this is to note that \eqref{eq:plane3} implies that the Gramian matrix $G$ with entries $G_{ij} = \langle Pv^i, Pv^{j} \rangle$ is a small perturbation of $\frac{1}{2}I$ when $K$ is large enough, with $I$ the $m\times m$ identity matrix. The claimed independence follows at once from the characterization of independence in terms of the Gramian matrix. Alternatively, from $\eqref{eq:plane3}$ it follows that there is some $\alpha > 0$ such that $$|\text{Angle}(Pv^i, Pv^j) - \pi/2| \leq \alpha K^{-a}, \ \ i\neq j $$ $$ ||Pv^i|^2 - 1/2| \leq \alpha K^{-a}$$ which implies the claimed independence if $K$ is large enough (depending on the value of $a$ and $\alpha$). 

Since the elements of $\{Pv^1, ..., Pv^{k} \}$ are all vectors in $\R^m$ we must have $k \leq m$ and so dim $V \leq m$.  
	 
\end{proof}

%
%

The following lemma is the main result of this section. 

\begin{lemma}\label{lem:caseLemma} Let $\{\tau\}$ be a collection of finitely-overlapping $K^{-1}$-caps in $B^{d-1}(0,2)$ with $Ef = \sum_{\tau} Ef_{\tau}.$ If $K$ is sufficiently large then one of the following must occur. \begin{enumerate}[label=(\roman*)]
		\item There exists a uniform $\alpha > 0$ and an $m$-dimensional affine space $V$ such that every $\tau$ is contained in an $O(K^{-\alpha })$ neighborhood of $V$.  
		\item There are two $K^{-1}$-caps $\tau, \tau'$ for which $$\inf_{ \substack{ \xi, \bar{\xi} \in \tau \\ \omega, \bar{\omega} \in \tau' } } | M(\xi - \omega) \cdot (\bar{\xi} - \bar{\omega})| \geq AK^{-1}.$$
	\end{enumerate}
\end{lemma}

\begin{proof} Suppose that (ii) fails and let $\tau_0, \tau_1, ..., \tau_{k}$ be distinct caps in $B^{d-1}(0,2)$ intersecting the support of $f$. We can assume we can find such caps with $k \geq 2$ or else (i) is trivially true. After modulating $Ef$ we can also assume that $\tau_0$ is centered at the origin.
	
Pick $\eta^{i} \in \tau_{i}$ for $i = 1,...,k$. Since (ii) fails for each pair of caps $(\tau_0, \tau_i)$ we see that $\eta^i \in \mathcal{C}_{cK^{-1}}$ for each $i$, with the constant $c$ depending only on $d,A$. Since (ii) also fails for each pair $(\tau_i, \tau_j)$ when $i \neq j$ we see that $\eta^i - \eta^j \in \mathcal{C}_{cK^{-1}}$ as well. Then by Lemma \ref{lem:coneLemma} we conclude that \begin{equation}\label{eq:spanId} \eta^i \cdot M\eta^j = O(K^{-1}) \end{equation} for each $i,j$ (including $i = j$).

 Let $\mathcal{S} = \{\eta^1, ..., \eta^{k} \}$. We now construct the space $V$ via a sequence of spaces $$V_1 \subset V_2 \subset ... \subset  V_{k'} = V,$$ with $$V_{j} = \textrm{span} \{\eta^{i_1}, ... ,\eta^{i_j} \}$$ for some subset of $\mathcal{S}$. We note that $V$ will be a vector space since we have shifted $\tau_0$ to the origin. Fix a small parameter $a > 0$, to be determined below. We pick any $\eta^{i_1} \in \mathcal{S}$ with $|\eta^{i_1}| \geq K^{-a}$ and set $V_1 = \textrm{span}\{\eta^{i_1}\}$. If no such $\eta^{i_1}$ exists then all of the caps are contained in an $O(K^{-a})$ neighborhood of the origin and (i) follows with $V = \{0\}$. Now assume, by induction, that we have constructed $$V_{j-1} = \textrm{span}\{\eta^{i_1}, ... , \eta^{i_{j-1}} \}.$$ If there are any $\eta^{i_{j}} \in \mathcal{S} \backslash \{\eta^{i_1}, ..., \eta^{i_{j-1}} \}$ such that \begin{equation}\label{eq:lemmaCond} |\eta^{i_j}| \geq K^{-a} \ \ \text{ and } \ \ \text{Angle}(\eta^{i_j}, V_{j-1}) \geq K^{-a}\end{equation} we pick one and let $$V_{j} = \text{span}\{\eta^{i_1},..., \eta^{i_j} \}.$$ If no such $\eta^{i_j}$ exists we stop the procedure and let $V = V_{j-1}$. This process continues until there are no more $\eta \in \mathcal{S}$ satisfying \eqref{eq:lemmaCond}. Say this happens at step $k' + 1$, in which case $V = V_{k'}$ (note that there are at most $d-1$ steps).
 
 After possibly relabeling we assume that $V = \text{span}\{\eta^1, \eta^2, ..., \eta^{k'} \}.$ We now claim that $k' \leq m$. This implies (i) since (by construction) the centers of the remaining caps in $B^{d-1}(0,2)$ make an angle $O(K^{-a})$ with $V$ or are contained in a ball of radius $O(K^{-a})$ centered at the origin. To prove the claim, first note that if $\omega \in V \cap B^{d-1}(0,2)$ with $\omega = \sum_{i=1}^{k'}a_i \eta^{i}$ then we have  \begin{equation}\label{eq:spanId2} \omega \cdot M\omega = \sum_{i,j} a_{i}a_j (\eta^{i}\cdot M\eta^{j}).\end{equation} We claim that it suffices to show that \begin{equation}\label{eq:aBound}|a_i| \lesssim K^{1/2 - \sigma} \end{equation} for some small $\sigma > 0$ depending only on $d$. Indeed, if \eqref{eq:aBound} holds then by \eqref{eq:spanId2} and \eqref{eq:spanId} we have $\omega \cdot M\omega = O(K^{-2\sigma})$, and therefore  $$V \cap B^{d-1}(0,2) \subset \mathcal{C}_{cK^{-2\sigma}}$$ since $\omega$ was arbitrary. But then Lemma \ref{lem:tangentPlane} implies that $k' = \text{dim } V \leq m$ (provided $K$ is sufficiently large), as desired. 
 
We now prove \eqref{eq:aBound}. After applying an orthogonal transformation we may assume that $V$ is spanned by the standard basis vectors $\{e_1, ..., e_{k'}\}$ and that the $j$-th component of $\eta^i$ is 0 for $j > k'$. Without loss of generality we may view $\omega$ and the $\eta^i$ as elements of $\R^{k'}$, since only their first $k'$ components are nonzero. We let $B$ denote the matrix with columns given by the $\eta^i$. We also let $B_i (x)$ denote the matrix obtained by replacing the $i$-th column of $B$ by $x$ (that is, replacing $\eta^i$ by $x$). Then we have the following identity, which is essentially Cramer's rule: \begin{equation}\label{eq:cramer} \omega = \frac{1}{\det(B)}\sum_{i=1}^{k'} \det(B_i (\omega)) \eta^i. \end{equation} Indeed, $\det(B_i(\eta^j)) = 0$ if $i \neq j$, so this identity follows by expanding $\det(B_{i}(\omega))$ using $\omega = \sum_{j} a_j \eta^j$ and linearity. From \eqref{eq:cramer} it follows that $a_i = \frac{\det(B_i (\omega))}{\det(B)},$ and hence by Hadamard's inequality $$|a_i| \leq \frac{1}{|\det(B)|} |\omega| \prod_{j \neq i}|\eta^j| \lesssim \frac{1}{|\det(B)|}. $$ Finally, recall that by construction \begin{equation} \label{eq:angle}
\text{Angle}(\eta^i, V_{i-1}) \geq K^{-a} \ \text{ with } \ V_{i-1} = \text{span}\{\eta^1, ..., \eta^{i-1}\}, \end{equation} and $|\eta^i| \geq K^{-a}$. Therefore  $$|\det(B)| = |\eta^1 \wedge ... \wedge \eta^{k'}| \gtrsim K^{-a(d-1)} K^{-a(d-2)} $$ (using $k' \leq d-1$). Indeed, the first term comes from rescaling the $\eta^i$ to have length one, and the second term is a lower bound for the volume of any parallelepiped spanned by unit vectors satisfying \eqref{eq:angle}. It follows that $|a_i| \lesssim K^{2a(d-1)}$, and so \eqref{eq:aBound} follows if we choose $a < \frac{1}{4(d-1)}$. This completes the proof.

(We remark that the choice of $a$ is far from optimal. For example, taking into account the lengths of the $\eta^i$ when applying Hadamard's inequality shows that we can take $a$ to be larger than $\frac{1}{4(d-1)}$. However, the precise value of $a$ is not relevant for the proof below so we have chosen to not track it too carefully). 
 
\end{proof}

In the next section we will take $K = R^{\delta}$ for some $\delta = \delta(\epsilon)$. We are allowed to assume that $K \geq C_{\epsilon}$ by induction, and therefore we will always be able to assume $K$ is large enough that Lemma \ref{lem:caseLemma} applies.

\section{The broad-narrow argument} We now prove Theorem \ref{thm:restrictionThm} using a broad-narrow argument adapted from \cite{BG}, \cite{G}, \cite{DZ}. Fix $\epsilon > 0$ for the rest of the argument. Let $\delta > 0$ be another parameter with $\delta < \epsilon^2$ and set $$K = R^{\delta} \text{  and  } K_{1} = K^{\alpha},$$ where $\alpha$ is as in part (i) of Lemma \ref{lem:caseLemma}. We assume that $\delta$ is small enough such that $$ K \ll K_1.$$ Let $\mathcal{T}$ be a collection of finitely-overlapping $K^{-1}$-caps $\tau$ covering the support of $f$ and use a partition of unity to decompose $f = \sum_{\tau}f_{\tau}$ with $f_{\tau}$ supported in (a small dilate of) $\tau$. We also let $\{\theta \}$ be a collection of finitely-overlapping $K_{1}^{-1}$-caps covering the support of $f$. Then $f = \sum_{\theta}f_{\theta}$ as well.

On the spatial side we fix a collection $\mathcal{Q}$ of finitely-overlapping $K^2$-cubes that cover $B^{d}(0,R)$. Given $Q\in \mathcal{Q}$ we define its significant set $$\mathcal{S}_{p}(Q) = \{ \tau \in \mathcal{T} : \|Ef_{\tau} \|_{L^p (Q)} \geq \frac{1}{100 (\# \mathcal{T}) }  \|Ef\|_{L^p (Q)} \}.$$ Note that we have $$ \| \sum_{\tau \notin \mathcal{S}_{p}(Q) } Ef_{\tau}\|_{L^{p}(Q)} \leq \frac{1}{100}\|Ef\|_{L^{p}(Q)}, $$ and so we will always be able to absorb these error terms into the left-hand side of our estimates for $\|Ef\|_{L^p (Q)}$ below.   

Now fix a uniform constant $A > 1$ to be determined below.  We say that a $K^{2}$-cube $Q$ is \textit{narrow }and write $Q \in \mathcal{N}$ if there is an $(m+1)$-dimensional subspace $W$ such that $$\text{Angle}(G(\tau), W) \leq AK_{1}^{-1}$$ for all $\tau \in \mathcal{S}_{p}(Q)$, where $G(\tau)$ is the unit normal to the surface $\mathbb{H}$ above the center of $\tau$. If a cube $Q$ is not narrow then we say it is \textit{broad} and write $Q \in \mathcal{B}$. We of course have $$\|Ef\|^{p}_{L^p (B_R)} \leq \sum_{Q \in \mathcal{N}} \|Ef\|_{L^{p}(Q)}^{p} + \sum_{Q \in \mathcal{B}} \|Ef\|_{L^{p}(Q)}^{p},$$ and so it suffices to consider separately the cases when the broad and narrow terms dominate. 

\subsection{The broad case} We first consider the broad case. We will need to use the following lemma which is a consequence of Theorem \ref{thm:bilinear} and the fact that $Ef$ is essentially constant at scale one. We recall that two caps $\tau_1$ and $\tau_2$ are said to be \textit{strongly separated} if \eqref{eq:bilinearCondition} holds.  

\begin{lemma}\label{lem:bilinearLem} Suppose $f$ is supported in $B^{d-1}(0,2)$. Let $\tau_1$ and $\tau_2$ be two strongly separated $K^{-1}$-caps. Then $$ \sum_{Q \in \mathcal{B} } \|Ef_{\tau_1}\|^{\frac{p}{2}}_{L^p(Q)}\|Ef_{\tau_2}\|^{\frac{p}{2}}_{L^p(Q)} \leq  K^{O(1)}\|f\|^{p}_{L^{2}}$$ whenever $p \geq \frac{2(d+2)}{d}$. 
\end{lemma}

\noindent The proof of this lemma is contained in the proof of Proposition 3.1 in \cite{DZ}, though for completeness we include most of the argument.
 
\begin{proof}  We define $f_i (\xi) = e^{i x_i \cdot \xi + t_i \cdot(M\xi \cdot \xi) }f_{\tau_i} (\xi)$ for some choice of $(x_i,t_i) \in \R^d$. Let $\phi$ be a bump function on $\R^d$ with $\widehat{\phi} = 1$ in $B^{d}(0,2)$ and $\widehat{\phi}$ supported in $B^{d}(0,3).$ Note that $Ef_{i} = Ef_{i} \ast \phi$ for any choice of $(x_i, t_i)$ in the definition of $f_i$. 
	
We first fix a single $Q$. Decompose $Q$ as a union of lattice cubes $L_Q$ of side-length $\frac{1}{10}$. Then we may find $(x_i, t_i)$ as above and $L_{Q}^{\ast} \subset Q$ such that $$\|Ef_i \ast \phi\|_{L^{\infty}(Q)} \leq \|Ef_i \ast \phi\|_{L^{\infty} (L_Q^{\ast})}$$ for both $i = 1,2$. Then

$$\|Ef_{\tau_1}\|^{\frac{1}{2}}_{L^p(Q)}\|Ef_{\tau_2}\|^{\frac{1}{2}}_{L^p(Q)}  \leq K^{O(1)} \|Ef_{1}\ast \phi \|_{L^{\infty}(L_Q^{\ast})}^{\frac{1}{2}}\|Ef_{2}\ast\phi\|_{L^{\infty}(L_Q^{\ast})}^{\frac{1}{2}}.$$ 

\noindent We may pick our bump function $\phi$ so that $\phi$ decays rapidly outside $B^{d}(0,1)$ with $$\sup_{w \in B^d (z, 1) } \phi(w) \lesssim \phi(z) \ \ \ \ \text{ for any } z \in \R^d.$$ Therefore \begin{align*}\|Ef_{1}&\ast \phi\|_{L^{\infty}(L_Q^{\ast})}^{\frac{p}{2}}\|Ef_{2}\ast \phi\|_{L^{\infty}(L_Q^{\ast})}^{\frac{p}{2}} \\ &\lesssim \big( \int_{L_Q^{\ast}} \int_{\R^d}\int_{\R^d} |Ef_1 (z_1)| |Ef_2(z_2)| \phi(z_1 - z) \phi(z_2 - z)  dz_1 dz_2 dz\big)^{\frac{p}{2}} \\ & = C\big(  \int_{\R^d}\int_{\R^d} \int_{L_Q^{\ast}}|Ef_1 (z_1 -z )| |Ef_2(z_2 - z)| \phi(z_1) \phi(z_2) dz  dz_1 dz_2\big)^{\frac{p}{2}}. \end{align*} We now sum over $Q$. By Minkowsi's and H\"{o}lder's inequalities we have

\begin{align*}&\sum_{Q \in \mathcal{B}} \|Ef_{1}\ast \phi\|_{L^{\infty}(L_Q^{\ast})}^{\frac{p}{2}}\|Ef_{2}\ast \phi\|_{L^{\infty}(L_Q^{\ast})}^{\frac{p}{2}} \\ &\lesssim \bigg[ \int_{\R^d}\int_{\R^d} \big( \int_{B_R} |Ef_1 (z_1 -z )|^{\frac{p}{2}} |Ef_2(z_2 - z)|^{\frac{p}{2}} \phi(z_1)^{\frac{p}{2}} \phi(z_2)^{ \frac{p}{2} } dz  \big)^{\frac{2}{p}} dz_1 dz_2 \bigg]^{\frac{p}{2}}  \\ &\lesssim \sup_{z_1, z_2}\int_{B_R} |Ef_1 (z_1 -z )|^{\frac{p}{2}} |Ef_2(z_2 - z)|^{\frac{p}{2}} dz  \\&\lesssim  \sup_{z_1, z_2}\int_{B_R} |E\widetilde{f_1} (z )|^{\frac{p}{2}} |E\widetilde{f_2}(z)|^{\frac{p}{2}}  dz\end{align*} where $\widetilde{f_{i}}$ is a modulation of $f_i$ that depends on $z_i$. Note that $$\|\widetilde{f_{i}}\|_{L^2} = \|f_{\tau_i}\|_{L^2}.$$ Since $\widetilde{f_{i}}$ is still supported in $\tau_i$ and the pair $(\tau_1, \tau_2)$ is strongly separated, we may apply Theorem \ref{thm:bilinear} to conclude that $$\sum_{Q \in \mathcal{B}}\|Ef_{1}\ast \phi\|_{L^{\infty}(L_Q^{\ast})}^{\frac{p}{2}}\|Ef_{2}\ast \phi\|_{L^{\infty}(L_Q^{\ast})}^{\frac{p}{2}} \leq K^{O(1)}\|f\|_{L^{2}}^{p},$$ which completes the proof.   
	
\end{proof}

Let $Q$ be a broad cube and first suppose that there is no strongly separated pair of caps in $\mathcal{S}_{p}(Q)$. Then by Lemma \ref{lem:caseLemma} there exists an $m$-dimensional affine space $V$ such that $\tau \subset N_{cK_{1}^{-1}}(V)$ for all $\tau \in \mathcal{S}_{p}(Q).$ But this forces the directions $G(\tau)$ to be in an $O(K_{1}^{-1})$ neighborhood of the $(m+1)$-plane $W$ in $\R^d$ given by scalar multiples of vectors in $G_{0}(V)$, where $$G_{0}(\omega) = |(\omega, - 1)|G(\omega)$$ (note that the angle between $G(\omega_1)$ and $G(\omega_2)$ is proportional to the distance $|\omega_1 - \omega_2|$ if the centers of the caps are $O(K^{-1})$-separated). Therefore $Q \in \mathcal{N}$, assuming we have chosen $A$ appropriately depending only on the constant from Lemma \ref{lem:caseLemma}. Since we are assuming $Q \in \mathcal{B}$ this cannot happen and so there must be two strongly separated caps $\tau_1, \tau_2 \in \mathcal{S}_{p}(Q)$. By the definition of $\mathcal{S}_p(Q)$ we then have $$\|Ef\|_{L^{p} (Q)} \leq K^{O(1)}\|Ef_{\tau_1}\|^{\frac{1}{2}}_{L^p(Q)}\|Ef_{\tau_2}\|^{\frac{1}{2}}_{L^p(Q)}.$$ The pair $(\tau_1, \tau_2)$ depends on $Q$, but we may make this estimate uniform by summing in $\ell^p$ over all possible strongly separated pairs (note the number of such pairs is $O(K^{2(d-1)})$). We then apply Lemma \ref{lem:bilinearLem} to conclude that \begin{align*}\sum_{Q \in \mathcal{B}}\|Ef\|^{p}_{L^{p}(Q)} &\leq K^{O(1)} \sum_{\substack{ (\tau_1, \tau_2) \\ \text{strongly sep.} }} \sum_{Q\in\mathcal{B} } \|Ef_{\tau_1}\|^{\frac{p}{2}}_{L^p(Q)}\|Ef_{\tau_2}\|^{\frac{p}{2}}_{L^p(Q)}\\ &\leq CR^{\epsilon p}\|f\|^{p}_{L^{p}}  \end{align*}(provided $\delta = \delta(\epsilon)$ is chosen small enough, e.g. $\delta = \epsilon^{4}$). 

\subsection{The narrow case} We now estimate the contribution of the narrow cubes. Suppose $Q \in \mathcal{N}$ and let $W$ be an $(m+1)$-plane in $\R^d$ such that $$\text{Angle}(G(\tau),W) \leq AK_{1}^{-1}$$ for each $\tau \in \mathcal{S}_{p}(Q)$. Then there is an $m$-dimensional affine space $V$ in $\R^{d-1}$ such that $\tau \subset N_{cK_{1}^{-1}}(V)$ for each $\tau \in \mathcal{S}_{p}(Q)$. In particular we can take $$V = \{\omega \in \R^{d-1} : G_0(\omega) \in W \}.$$ We choose a minimal collection $\Theta_{V}$ of $\theta$ covering $N_{cK_{1}^{-1}}(V)$. Note that $\Theta_{V}$ contains $cK_{1}^{m}$ caps $\theta$. Applying flat decoupling and then H\"{o}lder's inequality we obtain \begin{align*}\|Ef\|_{L^{p}(Q)} &\leq CK_{1}^{m(\frac{1}{2} - \frac{1}{p}) } \big( \sum_{\theta\in \Theta_{V}} \|Ef_{\theta}\|_{L^{p}(w_{Q})}^{2} \big)^{\frac{1}{2}}  \\ &\leq CK_{1}^{m(1 - \frac{2}{p} )} \big( \sum_{\theta\in \Theta_{V}} \|Ef_{\theta}\|_{L^{p}(w_{Q})}^{p} \big)^{\frac{1}{p}} \\ &\leq CK_{1}^{m (1 - \frac{2}{p}) } \big( \sum_{\theta} \|Ef_{\theta}\|_{L^{p}(w_{Q})}^{p} \big)^{\frac{1}{p}}. \end{align*} Since $$\sum_{Q}w_{Q} \lesssim w_{B_R}$$ we can sum over $Q$ to conclude that \begin{equation}\label{eq:narrowEst}\big(\sum_{Q \in \mathcal{N}} \|Ef\|_{L^{p}(Q)}^{p}\big)^{\frac{1}{p}} \leq  CK_{1}^{m(1 - \frac{2}{p}) } \big( \sum_{\theta} \|Ef_{\theta}\|_{L^{p}(w_{B_R})}^{p} \big)^{\frac{1}{p}}. \end{equation}

\noindent We will now use induction on scales. By Proposition \ref{prop:parabolicRescaling}, for each $\theta$ we can find a function $g_{\theta}$ supported in $B^{d-1}(0,2)$ such that $\|f_{\theta}\|_{L^p} = K_{1}^{-\frac{(d-1)}{p}}\|g_{\theta}\|_{L^{p}}$ and such that $$\|Ef_{\theta}\|_{L^{p}(w_{B_R})} \leq K_{1}^{-(d-1) + \frac{d+1}{p} }\|Eg_{\theta}\|_{L^{p}(w_{B_{R/K_{1} }})}.$$ By induction on scales we then obtain $$\|Ef_{\theta}\|_{L^{p}(w_{B_R})} \leq C_{\epsilon} R^{\epsilon} K_{1}^{-\epsilon }K_{1}^{-(d-1)+ \frac{d+1}{p}  }K_{1}^{\frac{d-1}{p} }\|f_{\theta}\|_{L^{p}}.$$ After applying this argument for each $\theta$ we see from \eqref{eq:narrowEst} that $$\big(\sum_{Q \in \mathcal{N}} \|Ef\|_{L^{p}(Q)}^{p}\big)^{\frac{1}{p}} \leq C_{\epsilon} R^{\epsilon}K_{1}^{-\epsilon }K_{1}^{m(1 - \frac{2}{p}) } K_{1}^{-(d-1)+ \frac{d+1}{p}  }K_{1}^{\frac{d-1}{p} }\|f\|_{L^{p}}.$$ The induction closes provided \begin{equation}\label{eq:induction} m(1 - \frac{2}{p}) - (d-1)+ \frac{2d}{p} \leq 0, \end{equation} since we may assume $K$ is large enough that $C_{\epsilon} K_{1}^{-\epsilon} \leq 1.$ Note that \eqref{eq:induction} is equivalent to $$p \geq \frac{2(d-m)}{d-m-1}.$$ Some algebra shows that $$ \frac{2(d-m)}{d-m-1} \leq \frac{2(d+2)}{d} $$ if and only if $$m \leq \frac{d}{2} -1.$$ We have assumed this is true for $m$, and so the narrow case of Theorem \ref{thm:restrictionThm} follows.

\vspace{3mm}

\begin{remark}\label{rmk:failure} In the narrow case above we have used flat decoupling in dimension $m$. This has nothing to do with the curvature of $\mathbb{H}$ and is true for any extension operator $E'f$ when $f$ is supported in a thin neighborhood of an $m$-plane. If one instead uses the stronger $\ell^2$ decoupling result proven by Bourgain and Demeter in \cite{BD2} there is no gain in our argument, since this still leads to a loss of $K_{1}^{m(\frac{1}{2} - \frac{1}{p})}$ in the first step. This is related to the fact that the surface $\mathbb{H}$ contains subsets which are affine spaces of dimension $m$, even though the curvature of $\mathbb{H}$ is nonzero. The $\ell^2$ decoupling does not distinguish the difference, since we can imagine that $Ef$ is supported in a small neighborhood of one of these affine spaces; in this case the $K_{1}^{m(\frac{1}{2} - \frac{1}{p} )}$ loss is sharp. 
	
We further elaborate on the last claim by considering the special case $d = 5, m = 2$. Note in this case $m = \frac{d-1}{2}$ and so our argument in the narrow case does not apply. Fix a $K^{2}$-cube $Q$ and suppose there is no pair of caps $(\tau_1, \tau_2)$ which are strongly separated and in $\mathcal{S}_{p}(Q)$. Then by Lemma \ref{lem:caseLemma} the support of $f$ must be contained in an $O(K_{1}^{-1})$-neighborhood of an $m$-plane $V$. If we assume there is at least one significant $\tau \in \mathcal{S}_{p}(Q)$ that contains the origin then from the proof of Lemma \ref{lem:caseLemma} we see that $V\cap B^{4}(0,2)$ can be taken to be a subset of the surface $\mathcal{C}$ defined in Section 2. Moreover $V$ can be assumed to be a vector space.   

Let $\{v, u\}$ be an orthonormal basis for $V$. Since $$v - u \in V\cap B^{4}(0, 2) \subset \mathcal{C}$$ the argument in Lemma \ref{lem:coneLemma} implies that $Mv \cdot u = 0$ and hence $Mu \cdot v =0$. We also know by hypothesis that $Mv \cdot v = 0$ and $Mu \cdot u = 0$. Therefore $\{v, u, Mv, Mu\}$ is an orthonormal basis for $\R^{4}$ with $V^{\perp} = \text{span}\{Mv, Mu \}.$ Now let $A$ be the orthonormal matrix with inverse $$A^{-1} = \begin{bmatrix} v & u & Mv & Mu \end{bmatrix}, $$ so that $A$ maps $V$ to the 2-plane determined by $\eta_{3} = 0$ and $\eta_{4} = 0$. Applying the change of coordinates determined by $A$ shows that $$\|Ef\|_{L^{p}(Q)} = \|\widetilde{E}f_{A}\|_{L^{p}(Q_{A})}  $$ where $f_A$ is the natural transform of $f$ and $\widetilde{E}$ is the extension operator with phase $$x\cdot \eta + t(M_{A} \eta \cdot \eta ),  $$ where $$ M_{A} = (A^{-1})^{T} M A^{-1} = \begin{bmatrix} 0 & 0 & 1 & 0 \\ 0 & 0 & 0 & 1 \\ 1 & 0 & 0 & 0 \\ 0 & 1 & 0 & 0
\end{bmatrix}. $$ In particular $\widetilde{E}$ is the extension operator associated to the hyperbolic surface $$\mathcal{H} = \{\eta \in \R^5 : \eta_5 = \eta_1 \eta_3 + \eta_2 \eta_4 \}. $$ Since $f$ is supported in a $K_{1}^{-1}$-neighborhood of $V$ it follows that $f_A$ is supported in a $K_{1}^{-1}$-neighborhood of the 2-plane where $\eta_3 = 0, \eta_4 = 0$. As a consequence $\widehat{\widetilde{E}f_{A} }$ is supported in a $K_{1}^{-1}$ neighborhood of the 2-plane $$V_A = \{\eta \in \R^5 : \eta = (\eta_1, \eta_2, 0, 0 ,0, 0) \}.$$ Note that $V_A \subset \mathcal{H}$ and therefore we can choose $f$ so that the loss of $K_{1}^{2(\frac{1}{2} - \frac{1}{p})}$ in our first decoupling step is sharp for general $f$. This can be seen for example by taking $f$ so that $\widehat{\widetilde{E}f_{A}}$ is essentially the indicator function of $V_A \cap B^5 (0, 2)$.

One is tempted to now exploit the non-isotropic scaling symmetry $$(\eta_1, \eta_2, \eta_3, \eta_4, \eta_5) \rightarrow (\eta_1, \eta_2, K_1 \eta_3, K_1 \eta_4, K_1 \eta_5) $$ associated to $\mathcal{H}$ and then argue by induction on scales (since such a transformation will map the support of $f_A$ to a cube of side-length $O(1)$ but shrink the size of $Q_A$). This gives a favorable result for each individual $Q$, but remember that $V$ can vary depending on $Q$ and may not even be a vector space. We have not found a way to effectively deal with the contribution of different $V$, mainly because $K_{1}^{-1}$-neighborhoods of different $V$ can intersect in complicated ways and naive estimates give a loss in $K_1$ that is much too large to close the induction. A similar issue arises in higher dimensions when $d$ is odd and $m = \frac{d-1}{2}$.

\end{remark}

\vspace{3mm} 

\begin{remark}\label{rmk:BG} The idea of using a broad-narrow analysis to deduce linear restriction theorems from multilinear restriction theorems dates back to Bourgain and Guth in \cite{BG}. They prove restriction estimates for the (elliptic) paraboloid by using $k$-linear restriction  (\cite{BCT}) in the broad case and an induction procedure in the narrow case. Their argument works in a range of $p$ that is larger than what Tao proved in \cite{T} using bilinear restriction theorems. When $d = 3$ their methods also adapt to the hyperbolic surface $\mathbb{H}$ and prove Theorem \ref{thm:restrictionThm} in this case. If $d \geq 4$ is even their methods also prove Theorem \ref{thm:restrictionThm}, and indeed in even dimensions the result follows from their more general estimates for H\"{o}rmander-type operators with non-degenerate phases. In this case one can avoid any type of induction-on-dimension procedure in the range $p \geq \frac{2(d+2)}{d}$ by directly using the $k$-linear Bennet--Carbery--Tao estimate with $k = \frac{d}{2} + 1$, along with a flat decoupling and induction-on-scales argument. In the narrow case in odd dimensions this procedure is not as effective since one needs to use a smaller $k$. 

Recall that the intersection of $\mathbb{H}$ with a hyperplane can have zero Gaussian curvature. This complicates any induction-on-dimension procedure when compared to the elliptic case, where the intersection of a paraboloid with a hyperplane is a paraboloid of lower dimension. The case $d = 3$ for $\mathbb{H}$ is special since you can only lose curvature if the hyperplane is (almost) parallel to the diagonal $\xi_1 = \xi_2$ or the anti-diagonal $\xi_1 = -\xi_2$. In this case case one can instead exploit non-isotropic scaling symmetries of the operator to close the induction. We have not found a way to carry this argument out in higher dimensions, except in the localized setting summarized at the end of the previous remark. Note that when $d = 3$ there are only two `bad' directions (the diagonal or anti-diagonal), but when $d \geq 4$ there are infinitely many (any direction along the hyperbolic cone $\mathcal{C}$ defined above). This appears to be one of the key differences between the cases $d = 3$ and $d = 5$, for example.

\end{remark}

\begin{remark} In \cite{HiI} Hickman and Iliopoulou prove restriction estimates for generalized extension operators with phases which are smooth perturbations of $x\cdot \xi + t(\xi\cdot M\xi)$. It is likely that the bilinear method in this paper will extend to smooth perturbations of the hyperbolic paraboloid $\mathbb{H}$. Indeed, in \cite{L} Lee proves a generalized version of Theorem \ref{thm:bilinear0} that allows for phases $\phi$ which are smooth perturbations of $ \xi \cdot M\xi $. It is likely that a version of Lemma \ref{lem:caseLemma} holds, with the plane $V$ replaced by an $m$-dimensional manifold determined by $\phi$. Then the rest of the argument would follow as in Section 3, with only minor changes made. We do not pursue the details here.   
\end{remark}

\Addresses

\end{document}